\documentclass[a4paper,twoside,reqno]{amsart}

\usepackage[pagewise]{lineno}
\usepackage[utf8]{inputenc}

\usepackage{authblk}
\usepackage{hyperref}
\hypersetup{urlcolor=blue, citecolor=red}
\usepackage{amsmath,amsthm,amssymb,xcolor,bbm,mathrsfs,graphicx,float}
\usepackage[english]{babel}
\usepackage{fancyhdr}
\newcommand{\n}{\hspace*{-6pt}}

\newcommand{\E}{\mathbf{E}}
\newcommand{\N}{\mathscr{N}}

\DeclareMathOperator{\diag}{diag}

\makeatletter
\@namedef{subjclassname@1991}{2020 Mathematics Subject Classification}
\makeatother
\subjclass{91C20, 91D25, 91D30, 94C15}
\keywords{Social network, Bonabeau model, egalitarian society, competing model}

\title{Bonabeau model on fully occupied site graphs}
\author{Hsin-Lun Li}

\date{}
\email{hsinlunl@asu.edu}

\theoremstyle{definition}
\newtheorem{theorem}{Theorem}

\newtheorem{lemma}[theorem]{Lemma}

\begin{document}

\allowdisplaybreaks

\thispagestyle{firstpage}
\maketitle
\begin{center}
    Hsin-Lun Li
    \centerline{$^1$National Sun Yat-sen University, Kaohsiung 804, Taiwan}
\end{center}
\medskip

\begin{abstract}
    The Bonabeau model is a competing model where agents fight to maintain or change their positions. Originally studied on a finite lattice, in this model, one agent is randomly selected to move to a neighboring site chosen at random. If the neighboring site is vacant, the agent moves there. However, if the site is occupied, a fight ensues. If the agent wins, they switch places with the other agent; otherwise, they remain in their original position. We investigate the Bonabeau model on fully occupied site graphs and derive a critical bound for the stability of the egalitarian state applicable to all fully occupied connected site graphs. Furthermore, we develop a competing model where all fights end in finite time on all site graphs.
\end{abstract}

\section{Introduction}
The Bonabeau model consists of a finite set of agents occupying the sites of a two-dimensional square lattice with a linear dimension \(L\). Each agent occupies a single site, and the density of agents on the lattice is denoted by \(\rho\). At each time step, an agent is selected uniformly at random and chooses a neighboring site uniformly at random. If the neighboring site is unoccupied, the agent moves to that site. However, if the neighboring site is occupied, a fight occurs. Let agents \(i\) and \(j\) be the attacker and the defender, respectively. If agent \(i\) defeats agent \(j\), they switch sites. Otherwise, they remain unchanged. The probability \(Q_{ij}(t)\) that agent \(i\) beats agent \(j\) at time \(t\) is defined by a Fermi function:
$$Q_{ij}(t) = \frac{1}{1 + \exp[-\eta(h_i(t) - h_j(t))]},$$
where \(\eta > 0\) and \(h_i(t) \in \mathbb{R}\) indicates the power of individual \(i\) at time \(t\). The larger \(h_i(t)\) is, the more powerful individual \(i\) is. The initial power distribution of all individuals is assumed to be an egalitarian state where everyone has power 0. $\eta>0$ indicates that a more powerful individual is likelier to win a fight, whereas a less powerful individual is likelier to lose a fight. $Q_{ij} + Q_{ji} = 1$ implies that a fight ends with either a win or loss outcome.

There are two modes in the Bonabeau model: competition and relaxation~\cite{lacasa2006bonabeau}. Competition occurs when there is a fight. The winner gains 1 in power, while the loser loses \(F \geq 1\) in power. Relaxation occurs at all times, decreasing each individual's power by a factor of \(1 - \mu\), where \(\mu \in (0, 1)\). Interpreting mathematically, letting \( [n] = \{1, \ldots, n\} \) denote the collection of all individuals, for all \( i \in [n], \)
\[
h_i(t+1)=\begin{cases}
    (1-\mu) h_i(t) & \text{if agent } i \text{ does not involve in a fight at time } t, \\
    (1-\mu) (h_i(t) + 1) & \text{if agent } i \text{ wins the fight at time } t, \\
    (1-\mu) (h_i(t) - F) & \text{if agent } i \text{ loses the fight at time } t.
\end{cases}
\]

Here, we investigate the Bonabeau model on fully occupied site graphs, where fights occur continuously. Let \( G = (V, E) \) denote a simple undirected site graph, with vertex set \( V \) and edge set \( E \). Each vertex represents a site, and each edge indicates the connection between two sites. Denote
$$ \N_i(t) = \{j \in [n] : i \neq j \ \text{and}\ \text{agents}\ i\ \text{and}\ j\ \text{are adjacent in}\ G\} $$
as the collection of all site neighbors of individual \( i \) at time \( t \). At each time step, a pair of adjacent individuals is selected, and they have an equal chance of being the attacker or defender. In the Bonabeau model, the interactions between agents are fights, resulting in a win-or-loss outcome. In contrast, in the Deffuant model~\cite{lanchier2020probability}, the Hegselmann-Krause model~\cite{lanchier2022consensus}, the mixed Hegselmann-Krause model~\cite{mHK,mHK2}, and the imitation model~\cite{li2024imitation}, the interactions between agents lead them to move closer to each other in opinions.

An \emph{egalitarian society} is achieved if all individuals' power approaches the same value as time progresses. It is clear that all individuals have the same power if and only if the standard deviation $\sigma$ of their power $(h_i)_{i \in [n]}$ is 0. A \emph{Laplacian} $\mathscr{L}$ of a simple undirected graph $G = (V, E)$ is defined as
\[
\mathscr{L} = \diag(d_1, \ldots, d_{|V|}) - A,
\]
where
\[
\begin{array}{rcl}
\displaystyle d_i & \n=\n & \displaystyle\text{degree of vertex } v_i \text{ in } G,\vspace{2pt} \\
\displaystyle A_{ij} &\n =\n &\displaystyle \mathbbm{1}\{(v_i, v_j) \in E\}.
\end{array}
\]
We can derive another Laplacian of \( G \) by reordering the vertices, and the eigenvalues remain the same. The \emph{order} of a graph \( G \), \( |G| \), is the number of vertices in \( G \).

There is always a fight in the Bonabeau model on fully occupied site graphs. We develop a competing model without endless fights, even on fully occupied site graphs. Naturally, the weakest individuals cannot trigger a fight, and other stronger individuals do not want to attack them since they gain nothing from the fight. Mathematically, this means there is an absorbing state $-\ell$ such that an individual is out of fights if their power is in that state. Unlike the Bonabeau model, the competing model does not undergo relaxation processes. Furthermore, the system is assumed to be conservative in power, namely \( F = 1 \). Also, the constant parameter $\eta>0$ is replaced with a function of time, $\eta_t>0$. In essence, the mechanism of the competing model aligns with the original assumptions of the Bonabeau model, but without relaxation, and with \( F = 1 \), an absorbing state \( -\ell \) as a negative integer, and $\eta$ replaced with $\eta_t$. For all \( i \in [n], \)
\begin{equation}\label{competing model}
h_i(t+1) =
\begin{cases}
    h_i(t) & \text{if agent } i \text{ does not involve in a fight at time } t, \\
    h_i(t) + 1 & \text{if agent } i \text{ wins the fight at time } t, \\
    h_i(t) - 1 & \text{if agent } i \text{ loses the fight at time } t.
\end{cases}
\end{equation}

\section{Main results}
For the Bonabeau model on fully occupied connected site graphs, we derive a critical upper bound for $(1+F)\eta$ to achieve the stable egalitarian state $0 \in \mathbb{R}^n$ as $n \to \infty$.

\begin{theorem}\label{Thm:critical upper bound for stability}
   $4\mu/(1-\mu)$ is the critical upper bound for $(1+F)\eta$ to obtain the stable egalitarian state $0 \in \mathbb{R}^n$ as $n \to \infty$ for all fully occupied connected site graphs, i.e.,
\begin{itemize}
    \item $0 \in \mathbb{R}^n$ is stable as $n \to \infty$ for all fully occupied connected site graphs if $(1+F)\eta < 4\mu/(1-\mu)$, and
    \item there is some fully occupied connected site graph with an unstable egalitarian state $0 \in \mathbb{R}^n$ as $n \to \infty$ when $(1+F)\eta > 4\mu/(1-\mu)$.
\end{itemize}
\end{theorem}

Next, we move on to the competing model~\eqref{competing model} and substantiate that all fights end in finite time on all site graphs.

\begin{theorem}\label{Thm:finite time convergence of a competing model}
    In the competing model~\eqref{competing model}, all fights end in finite time for any site graph~$G$.
\end{theorem}

Theorem~\ref{Thm:finite time convergence of a competing model} indicates that no agents in nonabsorbing states encounter each other after some finite time. Therefore, an egalitarian society is unachievable in the competing model~\eqref{competing model} on a connected site graph with $n \geq 2$.

\section{Bonabeau model on fully occupied site graphs}

\begin{lemma}\label{lemma:egalitarian state}
    All individuals' power approaches $(1-\mu)(1-F)/\mu n$ if an egalitarian society is achieved on a fully occupied site graph $G.$
\end{lemma}
\begin{proof}
Letting $d_i=|\N_i(t)|$, $Q_{ij}=Q_{ij}(t)$, $h_i^\star=h_i(t+1)$ and $h_i=h_i(t).$ By mean field approximation,
\begin{align*}
     h_i^\star&=(1-\frac{d_i}{|E|})(1-\mu)h_i+\frac{1-\mu}{|E|}\sum_{j\in \N_i}\big[Q_{ij}(h_i+1)+Q_{ji}(h_i-F)\big]\\
    &=(1-\frac{d_i}{|E|})(1-\mu)h_i+\frac{1-\mu}{|E|}\sum_{j\in \N_i}\big[h_i+(1+F)Q_{ij}-F\big]\ \hbox{for all}\ i\in [n].
\end{align*}
Hence 
\begin{align*}
    \sum_{i\in [n]}h_i^\star&=\frac{1-\mu}{|E|}\sum_{i\in [n]}(|E|-d_i)h_i+\frac{1-\mu}{|E|}\bigg(\sum_{i\in [n]}d_ih_i+(1+F)|E|-F2|E|\bigg)\\
    &=\frac{1-\mu}{|E|}\bigg(|E|\sum_{i\in [n]}h_i-\sum_{i\in [n]}d_ih_i+\sum_{i\in [n]}d_ih_i+|E|(1-F)\bigg)\\
    &=(1-\mu)\bigg(\sum_{i\in [n]}h_i+1-F\bigg).
\end{align*}
Letting $h_t=\sum_{i\in [n]}h_i(t)/n$ and $a=(1-\mu)(1-F)/n,$ it turns out that
\begin{align*}
    h_{t+1}&=(1-\mu)h_t+a=(1-\mu)^{t+1}h_0+[(1-\mu)^t+\ldots+1]a\\
    &=(1-\mu)^{t+1}h_0+a\frac{1-(1-\mu)^{t+1}}{\mu}\to\frac{(1-\mu)(1-F)}{\mu n}\ \hbox{as}\ t\to\infty.
\end{align*}
So everyone's power approaches $(1-\mu)(1-F)/\mu n$ if an egalitarian society is achieved.
\end{proof}

Observe that all individuals' power approaches zero regardless of the initial power distribution as the population increases if an egalitarian society is achieved.

\begin{lemma}[\cite{das2003improved}]\label{lemma:largest eigenvalue of a Laplacian}
    The largest eigenvalue $\lambda_1$ of a Laplacian of the site graph $G$ satisfies
    \begin{equation*}
        \lambda_1\leq \max\{d_i+d_j-|N_i\cap N_j|: v_i,v_j\in V\ \hbox{and}\ (v_i,v_j)\in E\},
    \end{equation*}
    where $d_i$ is the degree of vertex $v_i$ in $G$ and $N_i=\{v_j\in V: (v_i,v_j)\in E\}$ is the collection of all neighboring sites of vertex $v_i.$
\end{lemma}
Observe that $d_i+d_j-|N_i\cap N_j|=|N_i\cup N_j|$, which implies $\lambda_1\leq |V|.$

\begin{lemma}\label{lemma:eigenvalues of the Jacobian}
    On a fully occupied site graph $G=(V,E)$, let $J$ be the Jacobian matrix of $h$ during an egalitarian society,  $\mathscr{L}$ be a Laplacian of graph $G$, $$a=\frac{(1+F)\eta}{4|E|},\ \hbox{and let}\ I\in\mathbb{R}^{n\times n}\ \hbox{be the identity matrix}.$$ Then, $v$ is the eigenvector of $\mathscr{L}$ corresponding to eigenvalue $b$ $\iff$ $v$ is the eigenvector of $J$ corresponding to eigenvalue $(1-\mu)(1+ab)$.   
\end{lemma}

\begin{proof}
Letting $d_i=|\N_i(t)|$, $h_i=h_i(t)$, $h_i^\star=h_i(t+1)$ and $Q_{ij}=Q_{ij}(t).$ During an egalitarian society, we derive
$\partial Q_{ij}/\partial h_i=\eta/4$ and $\partial Q_{ij}/\partial h_j=-\eta/4$, therefore
\begin{align*}
    &\frac{\partial h_i^\star}{\partial h_i}=(1-\frac{d_i}{|E|})(1-\mu)+\frac{1-\mu}{|E|}\bigg[1+(1+F)\frac{\eta}{4}\bigg]d_i=(1-\mu)\bigg[1+\frac{(1+F)\eta d_i}{4|E|}\bigg],\\
    &\frac{\partial h_i^\star}{\partial h_j}=-\frac{(1-\mu)(1+F)\eta}{4|E|}\ \hbox{for}\ j\in \N_i\ \hbox{and}\ 0\ \hbox{for}\ j\notin \N_i\cup \{i\}.
\end{align*}
Because of $J_{ij}=\partial h_i^\star/\partial h_j,$ $$J-\lambda I=(1-\mu)(a\mathscr{L}+I)-\lambda I=(1-\mu)a\big[\mathscr{L}+\frac{1}{a}(1-\frac{\lambda}{1-\mu})I\big].$$
Hence, $v$ is the eigenvector of $\mathscr{L}$ corresponding to eigenvalue $b=-[1-\lambda/(1-\mu)]/a$ $\iff$ $v$ is the eigenvector of $J$ corresponding to $\lambda=(1-\mu)(1+ab).$  
\end{proof}
Since a Laplacian is positive semidefinite, it follows from Lemmas~\ref{lemma:eigenvalues of the Jacobian} that the Jacobian \( J \) is positive definite, with the smallest eigenvalue and the largest eigenvalue \( (1-\mu) \) and \( (1-\mu)(1+\lambda_1 a) \). In particular, when the site graph \( G \) is complete, \(\lambda_1 = n\), implying that \( (1-\mu)(1+na) \) is the largest eigenvalue of the Jacobian \( J \). We derive the following lemma from the stability test on the Jacobian matrix $J$. Namely, check the dominant eigenvalue of $J.$

\begin{lemma}\label{lemma:stability of an egalitarian state}
    On a fully occupied site graph $G=(V,E)$, an egalitarian state is stable if $$(1-\mu)\bigg[1+\lambda_1 \frac{(1+F)\eta}{4|E|}\bigg]<1,\ \hbox{and unstable if}\ (1-\mu)\bigg[1+\lambda_1 \frac{(1+F)\eta}{4|E|}\bigg]>1.$$
\end{lemma}
Since all components of the site graph \( G \) are independent, we can assume that \( G \) is connected. A tree of order \( n \) is the minimal connected graph with \( n-1 \) edges. Lemma~\ref{lemma:largest eigenvalue of a Laplacian} indicates $\lambda_1\leq n.$ Hence, we derive the following lemma.

\begin{lemma}\label{lemma:upper bound}
    \[ \frac{\lambda_1}{|E|} \leq \frac{n}{n-1} \quad \text{on a fully occupied connected site graph } G. \]
\end{lemma}
Observe that $\lambda_1/|E|=O(1)$ as $n\to\infty.$

\begin{lemma}
    Assume one of the following holds:
    \begin{itemize}
        \item The site graph $G$ is a fully occupied path.
        \item The site graph $G$ is fully occupied and complete.
    \end{itemize}
    Then, the egalitarian state $0\in \mathbb{R}^n$ is stable as $n \to \infty$.
\end{lemma}

\begin{proof}
    By Lemma~\ref{lemma:egalitarian state}, the egalitarian state occurs when all individuals' power approaches zero as $n \to \infty$. According to Lemma~\ref{lemma:largest eigenvalue of a Laplacian}, $\lambda_1 \leq 4$ on a path and $\lambda_1 = n$ on a complete graph. Therefore, $\lambda_1/|E| \to 0$ as $n \to \infty$.

    Consequently via Lemma~\ref{lemma:stability of an egalitarian state},
    \[
    (1-\mu)\left[1 + \lambda_1 \frac{(1+F)\eta}{4|E|}\right] \to 1-\mu < 1 \quad \text{as} \quad n \to \infty.
    \]
    Thus, the egalitarian state $0 \in \mathbb{R}^n$ is stable as $n \to \infty$.
\end{proof}
Observe that $\lambda_1=n$ if the site graph $G$ is a star. We derive the following lemma by Lemma~\ref{lemma:stability of an egalitarian state}.

\begin{lemma}\label{lemma:star graph}
    Assume that the site graph $G$ is a fully occupied star. Then, $0\in \mathbb{R}^n$ is stable if $$(1+F)\eta<\frac{4\mu}{1-\mu},\quad \hbox{and unstable if}\quad (1+F)\eta>\frac{4\mu}{1-\mu}\quad \hbox{as}\quad n\to\infty.$$
\end{lemma}

\begin{proof}[\bf Proof of Theorem~\ref{Thm:critical upper bound for stability}]
    Following from Lemmas~\ref{lemma:stability of an egalitarian state} and~\ref{lemma:upper bound}, we have
$$\liminf_{n \to \infty} \frac{4\mu|E|}{\lambda_1(1-\mu)} \geq \liminf_{n \to \infty} \frac{4\mu(n-1)}{(1-\mu)n} = \frac{4\mu}{1-\mu}.$$
Via Lemmas~\ref{lemma:egalitarian state} and~\ref{lemma:star graph}, this completes the proof.

\end{proof}

\section{Developing a model without endless fights}

\begin{lemma}\label{lemma:submartingale}
    Let $Z_t=\sum_{i,j\in [n]}(h_i(t)-h_j(t))^2$ and $A_t$ be the event of a fight at time $t$. Then, $(Z_t)_{t\geq 0}$ is a submartingale. In particular, $$\E[Z_{t+1}-Z_{t}]\geq 4n \mathbf{P}(A_t).$$
\end{lemma}

\begin{proof}
    If there is no fight at time~$t$, then $Z_t=Z_{t+1}$. Else, there is a fight between some agents $p$ and $q$. Let $h_i=h_i(t)$ and $h_i^\star=h_i(t+1)$. For $(h_i-h_j)^2-(h_i^\star-h_j^\star)^2$, we consider the following cases:
    \begin{enumerate}
        \item Only one of $i$ and $j$ belongs to $\{p,q\}$,\vspace{2pt}
        \item $\{i,j\}=\{p,q\}$.
    \end{enumerate}
Observe that 
$$\begin{array}{rcl}
    \displaystyle(h_i-h_p)^2-(h_i-h_p^\star)^2&\n= \n&(h_p^\star-h_p)^2+2(h_i-h_p^\star)(h_p^\star-h_p),\vspace{2pt}  \\
    \displaystyle (h_i-h_q)^2-(h_i-h_q^\star)^2&\n=\n &(h_q^\star-h_q)^2+2(h_i-h_q^\star)(h_q^\star-h_q). 
\end{array}$$
Since $h_p^\star-h_p=\pm 1\iff h_q^\star-h_q=\mp 1$, 
\begin{align*}
    &(h_i-h_p)^2-(h_i-h_p^\star)^2+(h_i-h_q)^2-(h_i-h_q^\star)^2=-2(h_p^\star-h_p)(h_p-h_q)-2,\\
    &(h_p-h_q)^2-(h_p^\star-h_q^\star)^2=-4(h_p^\star-h_p)(h_p-h_q)-4.
\end{align*}
Observe that
   \begin{align*}
       Z_t-Z_{t+1}&=2\bigg\{\sum_{i\in [n]-\{p,q\}}\left[(h_i-h_p)^2-(h_i-h_p^\star)^2+(h_i-h_q)^2-(h_i-h_q^\star)^2\right]\\
       &\hspace{1cm}+(h_p-h_q)^2-(h_p^\star-h_q^\star)^2\bigg\}\\
       &=-4n[(h_p^\star-h_p)(h_p-h_q)+1].
   \end{align*} 
Thus,
\begin{align*}
    \E[Z_t-Z_{t+1}|A_t]&=-4n(\E[(h_p-h_q)(2Q_{pq}-1)|A_t]+1)\leq -4n,\\
    \E[Z_{t+1}-Z_t]&=\E[Z_{t+1}-Z_t|A_t]\mathbf{P}(A_t)\geq 4n\mathbf{P}(A_t).
\end{align*}
\end{proof}
The submartingale~$Z_t$ is the key to substantiating the finite time convergence of the competing model~\eqref{competing model}.

\begin{proof}[\bf Proof of Theorem~\ref{Thm:finite time convergence of a competing model}]
Since $-\ell \leq h_i = -\sum_{j \in [n] \setminus \{i\}} h_j \leq (n-1)\ell$ for all $i \in [n]$, $Z_t$ is a bounded submartingale. It follows from the martingale convergence theorem that $Z_t$ almost surely converges to some random variable $Z_\infty$ with finite expectation. Assume by contradiction that fights are endless. It follows from Lemma~\ref{lemma:submartingale} that there exists an increasing sequence $(t_k)_{k \geq 0}$ such that $A_{t_k}$ occurs with positive probability. By the finiteness of the site graph~$G$, $P(A_{t_k}) \geq 1/\binom{|G|}{2}$ for all $k \geq 0$. Hence, $\mathbb{E}[Z_{t_k+1} - Z_{t_k}] \nrightarrow 0$ as $k \to \infty$, which contradicts $Z_{t+1} - Z_t \to 0$ as $t \to \infty$. 
\end{proof}

\section{Simulations}
In the numerical analysis for Theorem~\ref{Thm:critical upper bound for stability}, we consider 100 agents fully occupying the star graph. Setting $\eta = 1$, we compare three cases of $F$: $F = 1$, $F = 1.5$, and $F = 3$, and observe the variation of the standard deviation $\sigma$ with respect to $\mu$. In Figure~\ref{fig:StarGraph}, the lines in black, red, and blue represent the cases $F = 1$, $F = 1.5$, and $F = 3$, respectively. It can be observed that the critical value of $\mu$ required to achieve $\sigma = 0$ increases as $F$ grows.

\begin{figure}[H]
    \centering
    \includegraphics[width=0.7\textwidth]{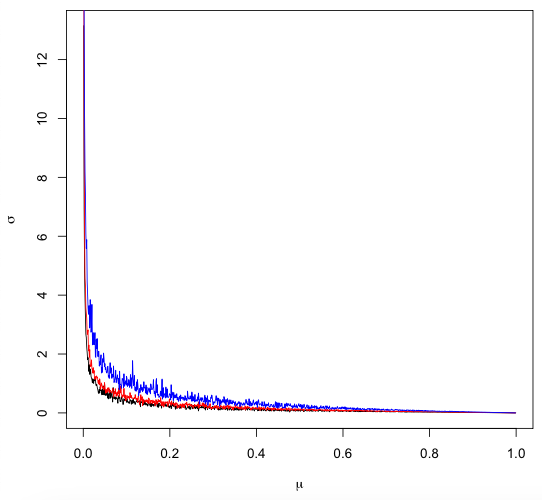}  
    \caption{On a fully occupied star graph}
    \label{fig:StarGraph}
\end{figure}

\section{Statements and Declarations}
\subsection{Competing Interests}
The author is partially funded by NSTC grant.

\subsection{Data availability}
No associated data was used.

\end{document}